\documentclass[11pt,reqno]{amsart}

\usepackage{amsmath, amsfonts, amsthm, amssymb, stmaryrd, color, cite}
\usepackage{parskip}
\usepackage{latexsym,hyperref}

\usepackage{amsmath,amssymb,amsfonts}
\usepackage{amsthm}
\usepackage{cite}
\usepackage{amsmath, amsfonts, amsthm, amssymb, stmaryrd, color, cite}

\usepackage{amsthm}
\usepackage{cite}
\usepackage{color}
\usepackage[dvipsnames]{xcolor}

\author{}
\newtheorem{theorem}{Theorem}
\newtheorem{definition}{Definition}
\newtheorem{lemma}{Lemma}

\title[\emph{LIL for a Class of SPDEs}]
{The Law of the Iterated Logarithm for a Class of SPDEs}
\date{}
\author[P. Fatheddin]{Parisa Fatheddin}
\address{Department of Mathematics, University of Pittsburgh, Pittsburgh, 15260, USA.}
\email{PAF49@pitt.edu}


\begin{document}
\maketitle
\begin{abstract}
After establishing the moderate deviation principle by the Classical Azencott method, we prove the Strassen's compact law of the iterated logarithm (LIL) for a class of stochastic partial differential equations (SPDEs). As an application, we obtain this type of LIL for two population models known as super-Brownian motion and Fleming-Viot process. In addition, the classical LIL is shown for the class of SPDEs and the two population models.
\end{abstract}

\section{Introduction}

Large deviations has noticeably become an active area of research with applications in queues, communication theory, exit problems and statistical mechanics. It is the study of very rare events that have probability tending to zero exponentially fast and its goal is to determine the exact form of this rate of convergence. Another closely related area of study is moderate deviations, which is proved for events that have probability going to zero at a rate slower than that of large deviations but faster than the rate for central limit theorem. An important application of large and moderate deviations is the law of the iterated logarithm (LIL). Beginning with J. Deuschel, D. Stroock \cite{Stroock} (Lemma 1.4.3), a notable number of authors have used this connection. For instance, P. Baldi \cite{Baldi2}, G. Divanji, K. Vidyalaxmi \cite{Divanji}, B. Jing, Q. Shao, Q. Wang \cite{Jing}, and A. Mogul'skii \cite{Mogulskii} applied their large deviation principle (LDP) to prove LIL; whereas, Y. Chen, L. Liu \cite{Y.Chen} and R. Wang, L. Xu \cite{Asymptotic} applied their result on moderate deviation principle (MDP).\\
 LIL has useful applications in fields such as finance (see for example \cite{Gao, Zinchenko}). There are different forms of LIL in the literature: Classical LIL, Strassen's Compact LIL, Chover's type, and Chung's type, which inherited names from the authors who introduced them; namely, A. Khintchine \cite{Khintchine}, V. Strassen \cite{Strassen}, J. Chover \cite{Chover}, and K. Chung \cite{Chung}, respectively. In section two, we provide a description of each type of LIL. For a more detailed introduction and history on each type we recommend \cite{Bingham}. As one may observe from the literature, every type of LIL may be derived from large and moderate deviations by the use of the Borel-Cantelli lemma. The most common form for this application is the Stassen's compact LIL. Here we prove the MDP by Azencott method and as an application establish the Strassen's compact LIL. We note that LDP and MDP for the class of SPDEs and the population models considered here were achieved in \cite{large} and \cite{moderate}, respectively by the weak convergence approach introduced by \cite{Budhiraja2, BDM}. Here the MDP by the Azencott method provides us the Freidlin-Wentzell inequality, which plays a major role in proving the Strassen's compact LIL. \\
To achieve the Strassen's compact LIL, one needs to show that the centered process multiplied by $1/\sqrt{2\log \log t}$ is relatively compact and then specify the set of limit points. Since our process is real-valued, we obtain that it is relative compact by proving its tightness property by a classical method. Moreover, to determine the set of limit points, we apply the result introduced by P. Baldi \cite{Baldi1} and implemented in \cite{Eddhabi, Nzi, Ouahra}. Other methods have also been applied by some authors to establish that their process is relatively compact. A. Dembo, T. Zajic \cite{Dembo2} and L. Wu \cite{Wu} prove this condition by showing that their process is totally bounded. A. Schied \cite{Schied}, attains the Strassen's compact LIL for super-Brownian motion (SBM) in all dimensions, $d\geq 1$, as a corollary to its moderate deviation result also given in \cite{Schied}. Similar to many other results in LIL derived from LDP or MDP, A. Schied uses the rate function in MDP to form the set of limit points. Since his rate function is a good rate function, he applies the compactness property of its level sets and Lemma 1.4.3 in \cite{Stroock}. Different from results in \cite{Schied}, here our proof of LIL for SBM relies on the more recent presentation of SBM by a stochastic PDE given by \cite{J.Xiong}.  To the best of our knowledge, LIL has not been proven for Fleming-Viot Process (FVP) in the literature. Also the Classical LIL for SBM and FVP would be new contributions. \\
We begin in section two with notations and spaces used throughout the paper and provide the main results. Also we offer some definitions and background on the Azencott method, LIL and the two population models. Section three is devoted to achieving the LDP, Strassen's compact LIL and Classical LIL for the class of SPDEs and in section four the results are applied to obtain the LDP and the two types of LIL for SBM and FVP.

\section{Notation and Main Results}
 Following the notation given in \cite{large, moderate}, we introduce the space used here as follows. Suppose $(\Omega, \mathcal{F}, P)$ is a probability space and $\{\mathcal{F}_{t}\}_{t\geq 0}$ is a family of non-decreasing, right continuous sub-$\sigma$-fields of $\mathcal{F}$ such that $\mathcal{F}_{0}$ contains all $P$-null subsets of $\Omega$. We denote $\mathcal{C}_{b}(\mathbb{R})$ to be the space of continuous bounded functions on $\mathbb{R}$ and $\mathcal{C}_{c}(\mathbb{R})$ to be composed of continuous functions in $\mathbb{R}$ with compact support. For $0<\beta \in \mathbb{R}$, let $\mathcal{M}_{\beta}(\mathbb{R})$ denote the set of $\sigma$-finite measures $\mu$ on $\mathbb{R}$ such that,
\begin{equation}\label{Mbeta}
\int e^{-\beta |x|} d\mu(x) <\infty,
\end{equation}
and let $\mathcal{P}_{\beta}(\mathbb{R})$ be the set of probability measures satisfying \eqref{Mbeta}. We endow these spaces with the topology defined by a modification of the usual weak topology: $\mu^{n}\rightarrow \mu$ in $\mathcal{M}_{\beta}(\mathbb{R})$ (respectively in $\mathcal{P}_{\beta}(\mathbb{R})$) iff for every $f\in \mathcal{C}_{b}(\mathbb{R})$,
\begin{equation*}
\int_{\mathbb{R}}f(x)e^{-\beta |x|} \mu^{n}(dx) \rightarrow \int_{\mathbb{R}}f(x)e^{-\beta |x|}\mu(dx).
\end{equation*}
 For $\alpha \in (0,1)$, consider the space $\mathbb{B}_{\alpha, \beta}$ composed of all functions $f:\mathbb{R}\rightarrow \mathbb{R}$ such that for all $m\in \mathbb{N}$, there exist $K_{1}, K_{2}>0$ with the following conditions:
\begin{eqnarray}
\left|f(y_{1})-f(y_{2})\right|&\leq& K_{1}e^{\beta m}|y_{1}-y_{2}|^{\alpha} , \hspace{.4cm} \forall |y_{1}|,|y_{2}| \leq m\\ \label{cond1}
|f(y)| &\leq& K_{2}e^{\beta |y|}, \hspace{.4cm} \forall y\in \mathbb{R},\label{cond2}
\end{eqnarray}
and with the metric,
\begin{equation*}
d_{\alpha,\beta}(u,v)=\sum^\infty_{m=1}2^{-m}(\|u-v\|_{m,\alpha,\beta}\wedge1),\qquad u,\;v\in\mathbb{B}_{\alpha,\beta},
\end{equation*}
 where
 \begin{equation*}
\|u\|_{m,\alpha,\beta}=\sup_{x\in\mathbb{R}}e^{-\beta|x|}|u(x)|+\sup_{y_{1}\neq y_{2}\\ \left|y_{1}\right|,\left|y_{2}\right|\leq m}\frac{|u(y_{1})-u(y_{2})|}{|y_{1}-y_{2}|^\alpha}e^{-\beta m}.
\end{equation*}
We refer to the collection of continuous functions on $\mathbb{R}$ satisfying \eqref{cond2} as $\mathbb{B}_{\beta}$, which is a Banach space with norm,
\begin{equation*}
\|f\|_{\beta}= \sup_{x\in \mathbb{R}} e^{-\beta |x|}|f(x)|.
\end{equation*}
The above space was used to match the setup in \cite{BDM}, where weak convergence approach was introduced to prove the large deviation principle and the technical difficulties in time discretization in classical Azencott method were avoided.\\
We now give a short introduction on the two population models under study. For more information we refer the reader to \cite{Dawson,Dynkin, Etheridge, J.Xiong2}. SBM is the continuous version of branching Brownian motion, the most classical and best known branching process where individuals reproduce according to Galton-Watson process. Since the population is set to evolve as a cloud in $\mathbb{R}^{d}$, it is a measure-valued process and because of its branching property, we associate a branching rate, denoted as $\varepsilon$. One of the common ways used to characterize SBM, denoted as $\{\mu_{t}^{\varepsilon}\}_{\varepsilon >0}$, is by the unique solution to the martingale problem given as: for all $f\in \mathcal{C}_{b}^{2}(\mathbb{R})$,
\begin{equation*}
M_{t}(f):= \left<\mu_{t}^{\varepsilon},f\right>-\left<\mu_{0}^{\varepsilon},f\right>
-\int_{0}^{t}\left<\mu_{s}^{\varepsilon},\frac{1}{2}\Delta f\right> ds,
\end{equation*}
is a square-integrable martingale with quadratic variation,
\begin{equation*}
\left<M(f)\right>_{t}= \varepsilon \int_{0}^{t} \left<\mu_{s}^{\varepsilon},f^{2}\right>ds,
\end{equation*}
 (see \cite{Etheridge} Section 1.5 for this formulation). Recently, \cite{J.Xiong} offered the following SPDE to characterize SBM,
\begin{equation}\label{SBM}
u_{t}^{\varepsilon}(y)=F(y)+ \int_{0}^{t}\int_{0}^{u_{s}^{\varepsilon}(y)} W(dads) + \int_{0}^{t} \frac{1}{2} \Delta u_{s}^{\varepsilon}(y)ds,
\end{equation}
where,
 \begin{equation}\label{SBM d}
u_{t}^{\varepsilon}(y)=\int_{0}^{y}\mu_{t}^{\varepsilon}(dx), \hspace{1cm} \forall y\in \mathbb{R},
\end{equation}

  $F(y)=\int_{0}^{y}\mu_{0}(dx)$, and $W$ is an $\mathcal{F}_{t}$-adapted space-time white noise random measure on $\mathbb{R}^{+}\times \mathbb{R}$ with intensity measure $dsda$.

The other model studied here is FVP, which observes the evolution of the population based on the genetic type of individuals. It is the continuous version of the step-wise mutation model, in which individuals move in $\mathbb{Z}^{d}$ according to a continuous time sample random walk. In FVP, the population is fixed throughout time with each individual having a gene type and every time a mutation occurs the individual changes gene type and moves to a new location. Therefore, the distribution of gene types is observed, making FVP a probability-measure valued process with mutation rate given by $\varepsilon$. More background on FVP can be found in \cite{Etheridge, Kurtz1, Kurtz2}. Similar to SBM, FVP can be given by the unique solution to the following martingale problem: for $f\in \mathcal{C}_{c}^{2}(\mathbb{R})$,
\begin{equation*}
M_{t}(f)= \left<\mu_{t}^{\varepsilon},f\right>-\left<\mu_{0}^{\varepsilon},f\right>-\int_{0}^{t} \left<\mu_{s}^{\varepsilon},\frac{1}{2}\Delta f\right>ds,
\end{equation*}
is a continuous square-integrable martingale with quadratic variation,
\begin{equation*}
\left<M_{t}(f)\right>= \varepsilon \int_{0}^{t}\left(\left<\mu_{s}^{\varepsilon},f^{2}\right>-\left<\mu_{s}^{\varepsilon},f\right>^{2}\right)ds,
\end{equation*}
(see \cite{Etheridge} Section 1.11 for more details on this formulation). An SPDE characterization of FVP was also made in \cite{J.Xiong}. There by using,
\begin{equation}\label{FVP d}
u_{t}^{\varepsilon}(y)= \mu_{t}^{\varepsilon}((-\infty,y]),
\end{equation}
FVP was proved to be given by,
\begin{equation}\label{FVP}
u_{t}^{\varepsilon}(y)= F(y) + \int_{0}^{t}\int_{0}^{1} \left(1_{a\leq u_{s}^{\varepsilon}(y)}-u_{s}^{\varepsilon}(y)\right)W(dsda) + \int_{0}^{t} \frac{1}{2}\Delta u_{s}^{\varepsilon}(y)ds,
\end{equation}
with the same description for $F(y)$ and noise as for SBM in \eqref{SBM}. Note that the main difference between (\ref{SBM}) and (\ref{FVP}) is in the second term. Hence, as in \cite{large, moderate}, we consider the following class of SPDEs and have SBM and FVP as special cases,
\begin{equation}\label{SPDE}
u_{t}^{\varepsilon}(y)=F(y)+\sqrt{\varepsilon}\int_{0}^{t}\int_{U}G(a,y,u_{s}^{\varepsilon}(y))W(dads)
+\int_{0}^{t}\frac{1}{2}\Delta u_{s}^{\varepsilon}(y)ds,
\end{equation}
where $(U,\mathcal{U},\lambda)$ is a measure space such that $L^{2}(U,\mathcal{U},\lambda)$ is separable, $F$ is a function of $\mathbb{R}$ and $u_{1},u_{2},u,y\in \mathbb{R}$. In addition, $G:U\times \mathbb{R}^{2}\rightarrow \mathbb{R}$ satisfies the following conditions,
\begin{eqnarray}
\int_{U}\left|G(a,y,u_{1})-G(a,y,u_{2})\right|^{2}\lambda(da)&\leq& K|u_{1}-u_{2}|,\label{con1}\\
\int_{U}\left|G(a,y,u)\right|^{2}\lambda(da)&\leq& K\left(1+|u|^{2}\right)\label{con2}.
\end{eqnarray}
We note that since the well-posedness of \eqref{SPDE} achieved in \cite{J.Xiong} was in dimension one, our results are limited to dimension one only. For moderate deviation principle, we consider the centered process,
\begin{equation}\label{centered}
v_{t}^{\varepsilon}(y)= \frac{a(\varepsilon)}{\sqrt{\varepsilon}} \left(u_{t}^{\varepsilon}(y)-u_{t}^{0}(y)\right),
\end{equation}
where $a(\varepsilon)$ satisfies,
\begin{equation}\label{conditions}
0\leq a(\varepsilon)\rightarrow 0,\hspace{.2cm} \frac{a(\varepsilon)}{\sqrt{\varepsilon}} \rightarrow \infty \text{ as } \varepsilon \rightarrow 0.
\end{equation}
One may observe that the speed, $a(\varepsilon)$ of moderate deviations is less than $\sqrt{\varepsilon}$, the speed for large deviations; hence, the term moderate is used. For the rate of decay we need the following controlled PDE version of \eqref{centered}, also referred to as the skeleton equation given by,
\begin{equation}\label{controlled}
S_{t}(h,y)=\int_{0}^{t}\int_{U}G(a,y,u_{s}^{0}(y))h_{s}(a)\lambda(da)ds + \frac{1}{2}\int_{0}^{t}\Delta S_{s}(h,y)ds,
\end{equation}
where $h_{s}\in L^{2}\left([0,1]\times U, ds\lambda(da)\right)$. It is not difficult to show that for every $h_{s}(\cdot)$ there is a unique solution to (\ref{controlled}). For MDP, the first term of SPDE \eqref{SPDE}, $F(y)$, is assumed to be in space, $\mathbb{B}_{\alpha,\beta_{0}}$ where $\beta_{0}\in (0,\beta)$. By the classical Azencott method, we prove the following MDP result.

\begin{theorem}\label{them1}
If $F\in \mathbb{B}_{\alpha,\beta_{0}}$ for $\alpha \in \left(0,\frac{1}{2}\right)$, then family $\{v^{\varepsilon}_{t}(y)\}_{\varepsilon>0}$ satisfies the LDP in $\mathcal{C}([0,1];\mathbb{B}_{\beta})$ with speed $a(\varepsilon)$ and  rate function,
\begin{equation}\label{rate1}
I(g)= \frac{1}{2}\inf\left\{\int_{0}^{1}\int_{U}\left|h_{s}(a)\right|^{2}\lambda(da)ds: g = S_{t}(h,y)\right\},
\end{equation}
which implies that family $\{u_{t}^{\varepsilon}(y)\}_{\varepsilon >0}$ obeys the MDP.
\end{theorem}
LDP by Azencott method was first introduced by \cite{Azencott, Priouret} and it may be described as follows. Suppose a family of random variables, $\{X_{1}^{\varepsilon}\}_{\varepsilon>0}$, on a Polish space, $E_{1}$, satisfies the LDP with rate function, $I_{1}: E_{1}\rightarrow [0,\infty]$. Then family $\{X_{2}^{\varepsilon}\}_{\varepsilon>0}$ on another Polish space, $E_{2}$, satisfies the LDP with rate function, $I_{2}(g):= \inf\{I_{1}(f):\Phi(f)=g\}$, if for any $R,\rho, a>0$, there exist $\eta>0$, and $\varepsilon_{0}>0$, such that for any $f\in E_{1}$ with $I_{1}(f)\leq a$ and any $\varepsilon \leq \varepsilon_{0}$,
\begin{equation}\label{exponentialinequality}
P\left(\|X_{2}^{\varepsilon}-\Phi(f)\|_{2}\geq \rho, \|X_{1}^{\varepsilon}-f\|_{1} <\eta\right) \leq \exp\left(-\frac{R}{\varepsilon^{2}}\right),
\end{equation}
where $\Phi: \{I_{1}\leq a\} \rightarrow E_{2}$ is continuous with respect to the topology of $E_{1}$ when restricted to sets $\{I_{1}\leq a\}$ for any $a>0$. Inequality \eqref{exponentialinequality} is referred to as the Freidlin-Wentzell inequality and in the setting of SPDEs, $\Phi(f)$ is the unique solution to the controlled PDE. For some examples of results on LDP for SPDEs by this method we refer the reader to \cite{Nualart, Burgers}. Below is the general definition of LDP. For more background on the large deviations theory we recommend \cite{Budhiraja1, Dembo, Dupuis}.

\begin{definition}[Large Deviation Principle (LDP)] The sequence $\left\{X^{\varepsilon}\right\}_{\varepsilon>0}$ satisfies the LDP on $\mathcal{E}$ with rate function $I$ if the following two conditions hold.\\
a. LDP lower bound: for every open set $U\subset \mathcal{E}$,
\begin{equation*}
-\inf_{x\in U} I(x) \leq \liminf_{\varepsilon \rightarrow 0} \varepsilon \log P(X^{\varepsilon} \in U),
\end{equation*}
b. LDP upper bound: for every closed set $C \subset \mathcal{E}$,
\begin{equation*}
\limsup_{\varepsilon \rightarrow 0} \varepsilon \log P(X^{\varepsilon} \in C) \leq -\inf_{x\in C}I(x).
\end{equation*}
\end{definition}

As for SBM and FVP, each model being a measure-valued process is denoted as $\{\mu_{t}^{\varepsilon}\}_{\varepsilon>0}$ with $\varepsilon$ being the branching rate or mutation rate based on context and is set to go to zero. For SBM, the Cameron-Martin space, $\mathcal{H}$ is used and for FVP we use $\tilde{\mathcal{H}}$, the space for which
conditions for $\mathcal{H}$ hold with $\mathcal{M}_\beta(\mathbb{R})$
replaced by the space of probability measures $\mathcal{P}(\mathbb{R})$, and with the additional assumption,
\begin{equation*}
\left<\mu_{t}^{0}, \frac{ d \left(\dot{\omega}_{t}-\frac{1}{2}\Delta^{*}\omega_{t}\right)}{d\mu_t^{0}}\right> = 0,
\end{equation*}
where for both population models the centered process for MDP is given by,
 \begin{equation}\label{w}
 \omega_{t}^{\varepsilon}(dy):= \frac{a(\varepsilon)}{\sqrt{\varepsilon}} \left(\mu_{t}^{\varepsilon}(dy)- \mu_{t}^{0}(dy)\right).
  \end{equation}
With the above notation, we obtain the following theorems.
\begin{theorem}
If $\omega_{0} \in \mathcal{M}_{\beta}(\mathbb{R})$  such that $F\in \mathbb{B}_{\alpha, \beta_{0}}$ then super-Brownian motion, $\{\mu_{t}^{\varepsilon}\}_{\varepsilon>0}$, obeys the MDP in $\mathcal{C}([0,1];\mathcal{M}_{\beta}(\mathbb{R}))$ with speed $a(\varepsilon)$ and rate function,
\begin{equation}\label{rate2}
  I(\omega)=  \left\{\begin{array} {ll}  \frac{1}{2} \displaystyle \int_{0}^{1}
  \int_{\mathbb{R}}\left|\frac{d\left(\dot{\omega} -\frac{1}{2}\Delta^{*}\omega_{t}\right)}{d\mu_{t}^{0}}y\right|^2 \mu_{t}^{0}(dy) dt
   & \mbox{\emph{if }} \mu_{t}^{0} \in \mathcal{H}_{\omega_{0}}, \\
   \infty
  & \mbox{\emph{otherwise}.}     \end{array}   \right.
 \end{equation}
\end{theorem}

\begin{theorem}
Let $\mathcal{P}_{\beta}(\mathbb{R})$ be the probability measure analog of $\mathcal{M}_{\beta}(\mathbb{R})$. If $\omega_{0} \in \mathcal{P}_{\beta}(\mathbb{R})$ such that $F
\in \mathbb{B}_{\alpha,\beta_{0}}$, then, Fleming-Viot process, $\{\mu^{\varepsilon}\}_{\varepsilon>0}$,
satisfies the MDP on $\mathcal{C}([0,1];
\mathcal{P}_{\beta}(\mathbb{R}))$ with speed $a(\varepsilon)$ and rate function,
\begin{equation}\label{rate3}
  I(\omega)=  \left\{\begin{array} {ll}  \frac{1}{2} \displaystyle \int_{0}^{1}
  \int_{\mathbb{R}}\left|\frac{d \left(\dot{\omega_{t}}-\frac{1}{2}\Delta^{*}\omega_{t}\right)}{d\mu_{t}^{0}}
y  \right|^2 \mu_{t}^{0}(dy) dt
   & \mbox{\emph{if }} \mu_{t}^{0} \in \tilde{\mathcal{H}}_{\omega_{0}}, \\
   \infty
  & \mbox{\emph{otherwise.}}     \end{array}   \right.
 \end{equation}
 \end{theorem}

As mentioned in the introduction, there are different types of LIL that appear in the literature. Below we provide a definition for each type.

\begin{definition}[Law of the Iterated Logarithm (LIL)]
Let $\{X_{j}\}_{j\geq 1}$ be an i.i.d. sequence of random variables with $S_{n}:= \sum_{j=1}^{n}X_{j}$.

i. Classical LIL: $\{X_{j}\}_{j\geq 1}$ is said to satisfy the classical LIL, also referred to as the Khintchine's LIL, if
\begin{eqnarray}
\limsup_{n\rightarrow \infty} \frac{S_{n}-n\mu}{\sigma\sqrt{2n\log\log n}}&=&1 \text{ a.s. }\label{eq1}\\
\liminf_{n\rightarrow \infty} \frac{S_{n}-n\mu}{\sigma\sqrt{2n\log\log n}}&=&-1 \text{ a.s. }\label{eq2}
\end{eqnarray}
for common mean $\mu$ and variance $\sigma^{2}$. We note that this version is also given by (\ref{eq1}) and (\ref{eq2}) with $S_{n}-n\mu$ replaced by $X_{n}$ with $\mu=0$ and $\sigma^{2}=1$. For examples of this form see for instance \cite{Asymptotic, Jing}.

ii. Strassen's Compact LIL: A class of functions $\mathcal{F}$ satisfies Strassen's compact LIL with respect to $\{X_{j}\}_{j\geq 1}$ if there is a compact set $J$ in $\ell_{\infty}(\mathcal{F})$ such that $\{X_{j}\}_{j\geq 1}$ is a.s. relatively compact and its limit set is $J$. See for example \cite{Baldi1,Dembo2, Wu}.

iii. Chover-type LIL: $\{X_{j}\}_{j\geq 1}$ satisfies Chover-type LIL if
\begin{equation}\label{chover}
\limsup_{n\rightarrow \infty} \left(\frac{|S_{n}|}{n^{1/\alpha}}\right)^{\frac{1}{\log \log n}}= e^{1/\alpha} \text{ a.s.}
\end{equation}
for $0<\alpha <2$. For examples of this form see \cite{Peng,Yamamuro}.

iv. Chung-type LIL: Let $S_{n}^{*} = max_{k\leq n}|S_{k}|$. Chung-type LIL for $\{X_{j}\}_{j\geq 1}$ holds if
\begin{equation}\label{chung}
\liminf_{n\rightarrow \infty} \frac{S_{n}^{*} \sqrt{\log\log n}}{\sqrt{n}}= \frac{\pi}{\sqrt{8}} \text{ a.s. }
\end{equation}
For results of this type see for example, \cite{Y.Chen, Mogulskii}.
\end{definition}

We are now ready to give our results on LIL. For $0<\varepsilon <1$, let,
\begin{equation}
Z_{t}^{\varepsilon}(y):= \frac{1}{\sqrt{2\varepsilon \log \log \frac{1}{\varepsilon}}} \left(u_{t}^{\varepsilon}(y)-u_{t}^{0}(y)\right),
\end{equation}
more precisely,
\begin{equation}\label{Zprocess}
Z_{t}^{\varepsilon}(y)= \frac{1}{\sqrt{2 \log \log \frac{1}{\varepsilon}}} \int_{0}^{t}\int_{U}G_{s}^{\varepsilon}\left(a,y,Z_{s}^{\varepsilon}(y)\right)W(dads)+ \int_{0}^{t} \frac{1}{2} \Delta Z_{s}^{\varepsilon}(y)ds,
\end{equation}
where,
\begin{equation}
G_{s}^{\varepsilon}\left(a,y,Z_{s}^{\varepsilon}(y)\right):= G\left(a,y, \sqrt{2\varepsilon \log \log \frac{1}{\varepsilon}} Z_{s}^{\varepsilon}(y) + u_{s}^{0}(y)\right).
\end{equation}
 Therefore, we have the process $\{v_{t}^{\varepsilon}(y)\}_{\varepsilon>0}$ from moderate deviations used in Theorem \ref{them1} with $a(\varepsilon)= 1/\sqrt{2\log \log (1/\varepsilon)}$. One can check that this fulfills the requirements of $a(\varepsilon)$ going to zero as $\varepsilon$ tends to zero, at a rate slower than $\sqrt{\varepsilon}$. Also based on conditions \eqref{con1} and \eqref{con2},
\begin{equation}\label{condition1}
\int_{U}\left|G_{s}^{\varepsilon}\left(a,y,Z^{\varepsilon}_{s,1}(y)\right)-
G_{s}^{\varepsilon}\left(a,y,Z^{\varepsilon}_{s,2}(y)\right)\right|^{2} \lambda(da) \leq
K_{3}\sqrt{2\varepsilon \log \log \frac{1}{\varepsilon}}\left|Z_{s,1}^{\varepsilon}(y)-Z_{s,2}^{\varepsilon}(y)\right|,
\end{equation}
\begin{equation}\label{condition2}
\int_{U}\left|G_{s}^{\varepsilon}\left(a,y,Z_{s}^{\varepsilon}(y)\right)\right|^{2}\lambda(da)\leq
K_{4}\left(1+ \left(2\varepsilon \log \log \frac{1}{\varepsilon}\right)Z_{s}^{\varepsilon}(y)^{2}+ e^{2\beta_{0}|y|}\right),
\end{equation}
where we have used the fact that $F\in \mathbb{B}_{\alpha,\beta_{0}}$, giving by condition \eqref{cond2},
\begin{equation}\label{u0con}
\left|u_{s}^{0}(y)\right|\leq K_{2}e^{\beta_{0}|y|}.
\end{equation}
We point out that the proof of the existence and uniqueness of solutions to SPDE, $\left\{u_{t}^{\varepsilon}(y)\right\}_{\varepsilon>0}$ given in \cite{J.Xiong} only relies on condition \eqref{con1}. Thus, we obtain the well-posedness of solutions to $Z_{t}^{\varepsilon}(y)$ and can use its mild solution given as,
\begin{equation}\label{mild}
Z_{t}^{\varepsilon}(y):= \frac{1}{\sqrt{2\log \log \frac{1}{\varepsilon}}}\int_{0}^{t}\int_{U}P_{t-s}G_{s}^{\varepsilon}(a,y,Z_{s}^{\varepsilon}(y))W(dads),
\end{equation}
where $P_{t-s}$ is the Brownian semigroup defined as $P_{t}f(y)=\int_{\mathbb{R}}p_{t}(x-y)f(x)dx$ with \\
$p_{t}(x-y)=\frac{1}{\sqrt{2\pi t}}e^{-\frac{|x-y|^{2}}{2t}}$. We prove the following results using Theorems 1-3.
\begin{theorem}\label{theorem1}
Process $\{Z_{t}^{\varepsilon}(y)\}_{0<\varepsilon<1}$ is relatively compact in $\mathcal{C}([0,1];\mathbb{B}_{\beta})$ and its set of limit points is exactly $L_{1}:= \left\{g \in \mathcal{C}\left([0,1];\mathbb{B}_{\beta}\right): I(g)\leq 1\right\}$ where $I(g)$ is defined by (\ref{rate1}).
\end{theorem}
Similarly using the MDP result for SBM and FVP, let
\begin{equation}\label{Ztilde}
\tilde{Z}_{t}^{\varepsilon}:= \frac{1}{\sqrt{2\varepsilon \log \log \frac{1}{\varepsilon}}} \left(\mu_{t}^{\varepsilon}(dy)-\mu_{t}^{0}(dy)\right).
\end{equation}
\begin{theorem}\label{theorem2}
Process $\{\tilde{Z}_{t}^{\varepsilon}\}_{0<\varepsilon<1}$ formed by SBM process, $\{\mu_{t}^{\varepsilon}\}_{\varepsilon>0}$ in (\ref{Ztilde}) is relatively compact in $\mathcal{C}([0,1];\mathcal{M}_{\beta}(\mathbb{R}))$ with set of limit points being
 $L_{2}:= \left\{\omega \in \mathcal{C}([0,1];\mathcal{M}_{\beta}(\mathbb{R})): I(\omega)\leq 1\right\}$, where $I(\omega)$ is given by (\ref{rate2}).
\end{theorem}
\begin{theorem}\label{theorem3}
Process $\{\tilde{Z}_{t}^{\varepsilon}\}_{0<\varepsilon<1}$ formed by FVP process, $\{\mu_{t}^{\varepsilon}\}_{\varepsilon>0}$ in (\ref{Ztilde}) is relatively compact in $\mathcal{C}([0,1];\mathcal{P}_{\beta}(\mathbb{R}))$ with set of limit points being
$L_{3}:= \left\{\omega \in \mathcal{C}([0,1];\mathcal{P}_{\beta}(\mathbb{R})): I(\omega)\leq 1\right\}$, where $I(\omega)$ is given by (\ref{rate3}).
\end{theorem}
Following the setup in \cite{Y.Chen, Wang,Asymptotic} we prove the classical LIL in our stochastic PDEs setting. We note that most results on classical LIL such as \cite{Davie, Jing} have been achieved for a sum of independent identically distributed random variables, where the Borel-Cantelli lemma is the main ingredient of the proof. We prove the classical LIL for the class of SPDEs and the population models by showing that for each respective family, $\{X^{\varepsilon}\}_{0<\varepsilon<1}$ of solutions,
\begin{eqnarray}
\limsup_{\varepsilon \rightarrow 0} \frac{\|X^{\varepsilon}-X^{0}\|_{\chi}}{\sqrt{2\varepsilon \log \log \frac{1}{\varepsilon}}}&=&1 \hspace{.4cm} \text{a.s.}, \label{limsup}\\
\liminf_{\varepsilon\rightarrow0} \frac{\|X^{\varepsilon}-X^{0}\|_{\chi}}{\sqrt{2\varepsilon \log \log \frac{1}{\varepsilon}}} &=&-1 \hspace{.4cm} \text{a.s.},
\label{liminf}
\end{eqnarray}
where $\chi$ is $\mathbb{B}_{\beta}$ for the class of SPDEs and it is $\mathcal{M}_{\beta}(\mathbb{R})$ and $\mathcal{P}_{\beta}(\mathbb{R})$ for SBM and FVP, respectively.

\section{LIL for Class of SPDEs}

We begin by proving Theorem 1, for which we derive the following Freidlin-Wentzell inequality based on our setting,
\begin{equation}\label{maininequality}
P\left(\left\|Z_{t}^{\varepsilon}-S_{t}(h,y)\right\|_{\beta}>\rho, \left\|\frac{1}{\sqrt{2\log \log \frac{1}{\varepsilon}}} W - h_{s}\right\|_{\infty}<\eta\right)
\leq \exp\left(-2R \log \log \frac{1}{\varepsilon}\right).
\end{equation}
Note that by Schilder's theorem, the LDP holds for Brownian sheet, $W$ with rate function denoted here by $\tilde{I}(\cdot)$.

\begin{lemma}
For every $a>0$, $S_{t}(\cdot,y):\{\tilde{I} \leq a\} \rightarrow \mathcal{C}([0,1];\mathbb{B}_{\beta})$ is continuous with respect to the uniform topology.
\end{lemma}

\begin{proof}
Let $a>0$ and $h_s,k_s\in L^{2}([0,1]\times U, ds\lambda(da))$ with $|h_{s}| \vee |k_{s}|\leq a$. By H\"older's inequality and \eqref{con2}, we have,
\begin{eqnarray}\label{Sinequality}
&&\|S_{t}(h_{s},x)-S_{t}(k_{s},x)\|_{\beta}^{2} \\
&=& \sup_{x\in \mathbb{R}} e^{-2\beta |x|} \left|\int_{0}^{t} \int_{U}P_{t-s}G(a,x,u_{s}^{0}(x))(h_{s}(a)-k_{s}(a))\lambda(da)ds\right|^{2}\nonumber\\
&\leq& \sup_{x\in \mathbb{R}} e^{-2\beta |x|} \left|\int_{0}^{t} \left(\int_{U}\left(P_{t-s}G(a,x,u_{s}^{0}(x))\right)^{2}\lambda(da)\right)^{1/2} \left(\int_{U}|h_{s}(a)-k_{s}(a)|^{2}\lambda(da)\right)^{1/2}ds\right|^{2}\nonumber\\
&\leq& t\sup_{x\in \mathbb{R}} e^{-2\beta |x|} \int_{0}^{t} \int_{\mathbb{R}} p_{t-s}^{2}(x-y)e^{2\beta |y|}dy \int_{\mathbb{R}}(1+|u^{0}_{s}(y)|^{2})e^{-2\beta |y|}dy \int_{U} |h_{s}(a)-k_{s}(a)|^{2}\lambda(da)ds.\nonumber
\end{eqnarray}
Observe that,
\begin{eqnarray}\label{pinequality}
\int_{0}^{t}\int_{\mathbb{R}} p_{t-s}^{2}(x-y)e^{2\beta |y|}dyds &=& \int_{0}^{t}\int_{\mathbb{R}} \frac{1}{2\sqrt{\pi (t-s)}}p_{\frac{t-s}{2}}(x-y)e^{2\beta |y|}dyds\nonumber\\
&\leq& \tilde{K}e^{2\beta |x|} \int_{0}^{t} \frac{1}{\sqrt{t-s}}ds= \tilde{K}\sqrt{t} e^{2\beta |x|}.
\end{eqnarray}
Thus, using \eqref{u0con} with the convention that $\beta_{0}<\beta$, we obtain,
\begin{equation*}
\|S_{t}(h_{s},x)-S_{t}(k_{s},x)\|_{\beta}^{2} \leq t^{3/2} \tilde{K} K_{2} \int_{U} \sup_{0\leq s\leq t} |h_{s}(a)-k_{s}(a)|^{2}\lambda(da),
\end{equation*}
where, by noting the domain $L^{2}([0,1]\times U, ds\lambda(da))$ of $h_{s}$ and $k_{s}$, we may apply the dominated convergence theorem to obtain the result.

\end{proof}

As shown in \cite{Nualart}, by an application of Girsanov's transformation theorem, to obtain inequality \eqref{maininequality}, it is sufficient to prove for all $h_{s}\in L^{2}\left([0,1]\times U, ds\lambda(da)\right)$, $R, \rho >0$, there exist $\eta>0, \varepsilon_{0}\in (0,1)$, such that for all $\varepsilon \in (0,\varepsilon_{0})$,
\begin{equation}\label{mainestimate}
P\left(\|Y^{\varepsilon}_{t}-S_{t}(h,y)\|_{\beta}>\rho, \left\|\frac{1}{\sqrt{2\log \log \frac{1}{\varepsilon}}}W\right\|_{\infty}<\eta \right)\leq \exp\left(-2R\log \log \frac{1}{\varepsilon}\right),
\end{equation}
with,
\begin{eqnarray}\label{Yepsilon}
Y_{t}^{\varepsilon}(y)&=& \frac{1}{\sqrt{2 \log \log \frac{1}{\varepsilon}}} \int_{0}^{t}\int_{U}G_{s}^{\varepsilon}(a,y,Y_{s}^{\varepsilon}(y))W(dads)+ \int_{0}^{t}\frac{1}{2}\Delta Y_{s}^{\varepsilon}(y)ds\nonumber \\
&&+ \int_{0}^{t}\int_{U}G_{s}^{\varepsilon}(a,y,Y_{s}^{\varepsilon}(y))h_{s}(a)\lambda(da)ds,
\end{eqnarray}
where the well-posedness of $Y^{\varepsilon}_{t}(y)$ may be verified following similar reasoning as in the proof of the well-posedness of $u_{t}^{\varepsilon}(y)$ in \cite{J.Xiong}. For our estimates, we use the following lemma, the proof of which is very similar to that of Lemma 1 in \cite{moderate} and it is thus omitted.

\begin{lemma}\label{Ylemma}
Suppose $Y_{t}^{\varepsilon}(y)$ is the unique solution to SPDE \eqref{Yepsilon}, then for every $p\geq 1, \varepsilon>0$ and $T>0$, there exists a positive constant $M_{1}$ such that,
\begin{equation}\label{M1}
\sup_{\varepsilon>0} \mathbb{E}\left(\sup_{0\leq t\leq T} \int_{\mathbb{R}}Y_{t}^{\varepsilon}(x)^{2}e^{-2\beta |x|}dx\right)^{p}\leq M_{1}.
\end{equation}
\end{lemma}

In order to obtain \eqref{mainestimate}, we apply a time discretization of $Y^{\varepsilon}_{t}$. For $n\in \mathbb{N}$, $i=0,1,...,n$, let $\Delta_{i}^{n}= \left[t_{i}^{n},t_{i+1}^{n}\right)$, where $t_{i}^{n}=iT/n$, then by the following two estimates we can achieve inequality \eqref{mainestimate}.
\begin{equation}\label{firstestimate}
P\left(\|Y^{\varepsilon}_{t}-Y^{\varepsilon}_{t_{i}^{n}}\|_{\beta}>\mu\right)\leq \exp\left(-2R\log \log \frac{1}{\varepsilon}\right),
\end{equation}
\begin{equation}\label{secondestimate}
P\left(\|Y^{\varepsilon}_{t}-S_{t}(h,y)\|_{\beta}>\rho, \left\|\frac{1}{\sqrt{2 \log \log \frac{1}{\varepsilon}}}W\right\|_{\infty}<\eta, \|Y^{\varepsilon}_{t}-Y^{\varepsilon}_{t_{i}^{n}}\|_{\beta}\leq \mu\right)
\leq \exp\left(-2R\log \log \frac{1}{\varepsilon}\right).
\end{equation}

\begin{lemma}\label{lemma1}
For all $R>0, \mu>0$ there exists $n_{0}\in \mathbb{N}$ such that for all $n\geq n_{0}$, and $\varepsilon \in (0,1)$,
\begin{equation}\label{lemma1r}
P\left(\|Y^{\varepsilon}_{t}-Y^{\varepsilon}_{t_{i}^{n}}\|_{\beta} >\mu \right)\leq \exp \left(-2R \log \log \frac{1}{\varepsilon}\right).
\end{equation}
\end{lemma}
\begin{proof}
For $n\in \mathbb{N}$, let $t\in \Delta_{i}^{n}$ and denote for $0<t_{1}<t_{2}$,
\begin{equation}\label{deltap}
\Delta p(t_{2},t_{1}):= p_{t_{2}-s}(x-y)-p_{t_{1}-s}(x-y).
\end{equation}
Then we have,
\begin{eqnarray*}
&&\sup_{0\leq t\leq 1}\sup_{y\in \mathbb{R}} e^{-\beta |y|}|Y^{\varepsilon}_{t}(y)-Y^{\varepsilon}_{t_{i}^{n}}(y)|\\
&\leq& \sup_{0\leq t\leq 1}\left|\sup_{y\in \mathbb{R}} \frac{e^{-\beta |y|}}{\sqrt{2\log \log \frac{1}{\varepsilon}}}\int_{t_{i}^{n}}^{t}\int_{U}P_{t-s}G_{s}^{\varepsilon}(a,y,Y_{s}^{\varepsilon}(y))W(dads)\right|\\
&&+ \sup_{0\leq t\leq 1}\left|\sup_{y\in \mathbb{R}}\frac{ e^{-\beta |y|}}{\sqrt{2 \log \log \frac{1}{\varepsilon}}} \int_{0}^{t_{i}^{n}}\int_{U} \int_{\mathbb{R}} \Delta p\left(t,t_{i}^{n}\right)G_{s}^{\varepsilon}(a,x,Y_{s}^{\varepsilon}(x))dxW(dads)\right|\\
&&+ \sup_{0\leq t\leq 1}\sup_{y\in \mathbb{R}} e^{-\beta |y|}\left|\int_{t_{i}^{n}}^{t}\int_{U} P_{t-s} G_s^{\varepsilon}(a,y,Y_{s}^{\varepsilon}(y))h_{s}(a)\lambda(da)ds \right|\\
&&+ \sup_{0\leq t\leq 1}\sup_{y\in \mathbb{R}} e^{-\beta |y|}\left|\int_{0}^{t_{i}^{n}}\int_{U}\int_{\mathbb{R}}\Delta p(t,t_{i}^{n}) G_s^{\varepsilon}(a,x,Y_{s}^{\varepsilon}(x))h_{s}(a)dx\lambda(da)ds\right|\\
&=& I_{1}+I_{2}+I_{3}+I_{4},
\end{eqnarray*}
leading to,
\begin{equation*}
P\left(\sup_{0\leq t\leq 1}\|Y_{t}^{\varepsilon}-Y_{t_{i}^{n}}^{\varepsilon}\|_{\beta}>\mu\right)
\leq  \sum_{i=1}^{n}\sum_{j=1}^{4}P\left(\sup_{t \in \Delta_{i}^{n}}I_{j}(t)>\frac{\mu}{4}\right).
\end{equation*}
Similar to estimates in \eqref{Sinequality} and noting the domain $L^{2}\left([0,1]\times U, ds\lambda(da)\right)$ of $h_{s}(a)$, we have by Lemma 2,
\begin{eqnarray*}
P\left(\sup_{t\in \Delta_{i}^{n}}I_{3}(t)>\frac{\mu}{4}\right)&\leq& \frac{16}{\mu^{2}}\mathbb{E}\sup_{t\in \Delta_{i}^{n}}\left|I_{3}(t)\right|^{2}\leq M_{1}\tilde{K}_{1}\frac{16}{\mu^{2}}\sup_{t\in \Delta_{i}^{n}}\left|t-t_{i}^{n}\right|^{2},\\
\text{and  } P\left(\sup_{t\in \Delta_{i}^{n}}I_{4}(t)>\frac{\mu}{4}\right)&\leq& \frac{16}{\mu^{2}}\mathbb{E}\sup_{t\in \Delta_{i}^{n}}\left|I_{4}(t)\right|^{2}\leq M_{1}\tilde{K}_{2}\frac{16}{\mu^{2}}\sup_{t\in \Delta_{i}^{n}}\left|t-t_{i}^{n}\right|^{2}.
\end{eqnarray*}
Then for any fixed $R>0, \varepsilon \in (0,1)$, there exists $n_{0}$ such that for all $n\geq n_{0}$,
\begin{equation*}
P\left(\sup_{t\in \Delta_{i}^{n}}I_{3}(t)>\frac{\mu}{4}\right)+ P\left(\sup_{t\in \Delta_{i}^{n}}I_{4}(t)>\frac{\mu}{4}\right)\leq \tilde{K}_{3}\frac{16}{\mu^{2}}\sup_{t\in \Delta_{i}^{n}}\left|t-t_{i}^{n}\right|^{2}\leq \exp\left(-2R\log \log \frac{1}{\varepsilon}\right).
\end{equation*}
 Continuing, we obtain,
\begin{flalign*}
&P\left(\sup_{0\leq t\leq 1}\|Y^{\varepsilon}_{t}-Y^{\varepsilon}_{t_{i}^{n}}\|_{\beta}>\mu\right)\\
&\leq \sum_{i=1}^{n} \left(P\left(\sup_{t \in \Delta_{i}^{n}}I_{1}>\frac{\mu}{4}\right)+ P\left(\sup_{t \in \Delta_{i}^{n}}I_{2}>\frac{\mu}{4}\right)\right)\\
&\leq  \sum_{i=1}^{n}\left[P\left(\sup_{t\in \Delta_{i}^{n}} \left|\sup_{y\in \mathbb{R}}\frac{e^{-\beta |y|}}{\sqrt{2\log \log \frac{1}{\varepsilon}}}\int_{t_{i}^{n}}^{t}\int_{U}
P_{t-s}G_{s}^{\varepsilon}(a,y,Y_{s}^{\varepsilon}(y))W(dads)\right|>\frac{\mu}{4}\right)\right.\\
& \left.+ P\left(\sup_{t\in \Delta_{i}^{n}} \left|\sup_{y\in \mathbb{R}}\frac{e^{-\beta |y|}}{\sqrt{2\log \log \frac{1}{\varepsilon}}}\int_{0}^{t_{i}^{n}}\int_{U}\int_{\mathbb{R}}\Delta p(t,t_{i}^{n})
G_{s}^{\varepsilon}(a,x,Y_{s}^{\varepsilon}(x))dxW(dads)\right|>\frac{\mu}{4}\right)\right]\\
&= \sum_{i=1}^{n}P\left(\sup_{t\in \Delta_{i}^{n}}J_{1}^{i}(t)>\frac{\mu\sqrt{2\log\log\frac{1}{\varepsilon}}}{4}\right)+ \sum_{i=1}^{n} P\left(\sup_{t\in \Delta_{i}^{n}}J_{2}^{i}(t)>\frac{\mu \sqrt{2\log\log \frac{1}{\varepsilon}}}{4}\right).
\end{flalign*}
Similar to \eqref{Sinequality} and \eqref{pinequality} and by Burkholder-Davis-Gundy inequality, \eqref{condition2} and Lemma 2,
\begin{eqnarray*}
\mathbb{E}\sup_{t\in \Delta_{i}^{n}} |J_{1}^{i}|^{2}&=& \mathbb{E} \sup_{t\in \Delta_{i}^{n}} \sup_{y\in \mathbb{R}}e^{-2\beta |y|} \left|\int_{t_{i}^{n}}^{t} \int_{U} P_{t-s}G_{s}^{\varepsilon}(a,y,Y_{s}^{\varepsilon}(y))W(dads)\right|^{2}\\
&\leq& \mathbb{E} \sup_{t\in \Delta_{i}^{n}} \sup_{y} e^{-2\beta |y|} \int_{t_{i}^{n}}^{t} \int_{U} \left(\int_{\mathbb{R}} p_{t-s}(x-y) G_{s}^{\varepsilon}(a,x,Y_{s}^{\varepsilon}(x))dx\right)^{2}\lambda(da)ds\\
&\leq& \mathbb{E}\sup_{t\in \Delta_{i}^{n}} \sup_{y} e^{-2\beta |y|} \int_{t_{i}^{n}}^{t} \int_{U}\int_{\mathbb{R}} p_{t-s}(x-y)^{2} e^{2\beta |x|}dx \int_{\mathbb{R}} G_{s}^{\varepsilon}(a,x,Y_{s}^{\varepsilon}(x))^{2}e^{-2\beta |x|} dx \lambda(da)ds\\
&\leq& K_{4}\tilde{K}_{4} \mathbb{E}\sup_{t\in \Delta_{i}^{n}} |t-t_{i}^{n}|^{1/2} \int_{\mathbb{R}} \left(1+ (2\varepsilon \log \log \frac{1}{\varepsilon})\sup_{t_{i}^{n}\leq s\leq t} Y_{s}^{\varepsilon}(x)^{2} + e^{2\beta_{0}|x|}\right)e^{-2\beta |x|}dx\\
&\leq& \tilde{K}_{5} M_{1} |t-t_{i}^{n}|^{1/2}.
\end{eqnarray*}
Following the same steps as above we find,
\begin{equation*}
\mathbb{E}\sup_{t\in \Delta_{i}^{n}}|J_{2}(t)|^{2} \leq \tilde{K}_{6} M_{1} |t_{i}^{n}|^{1/2}.
\end{equation*}
Notice that to obtain \eqref{lemma1r}, it is sufficient to prove,
\begin{equation}\label{k}
\sum_{i=1}^{n}\mathbb{E}\exp\left(\sup_{t\in \Delta_{i}^{n}}|J_{1}^{i}(t)|^{2}\right) + \sum_{i=1}^{n}\mathbb{E}\exp\left(\sup_{t\in \Delta_{i}^{n}}|J_{2}^{i}(t)|^{2}\right)\leq \tilde{C},
\end{equation}
for some positive constant, $\tilde{C}$. Inspired by the proof of Theorem 3.2 in \cite{Cerrai}, we write the left hand side of \eqref{k} as,
\begin{equation*}
\sum_{i=1}^{n}\mathbb{E} \lim_{k\rightarrow \infty} \sum_{p=0}^{k} \frac{1}{p!}\sup_{t\in \Delta_{i}^{n}}|J_{1}^{i}(t)|^{2p} + \sum_{i=1}^{n}\mathbb{E} \lim_{k\rightarrow \infty}\sum_{p=0}^{k} \frac{1}{p!} \sup_{t\in \Delta_{i}^{n}}|J_{2}^{i}(t)|^{2p}.
\end{equation*}
Then observing that \eqref{M1} holds for all $p\geq 1$, we may apply the Monotone convergence theorem to arrive at \eqref{k} and since the above estimates hold for any $\mu>0$, we obtain \eqref{lemma1r}.

\end{proof}

\begin{lemma}
For all $R>0, \rho>0, n\in \mathbb{N}$, there exist $\mu_{0}, \eta_{0}>0$ such that for all $\mu\leq \mu_{0}, \eta\leq \eta_{0}$, and $\varepsilon \in (0,1)$,
\begin{equation}\label{2ndequation}
P\left(\left\|Y^{\varepsilon}_{t}-S_{t}(h,y)\right\|_{\beta}>\rho, \left\|\frac{1}{\sqrt{2 \log \log \frac{1}{\varepsilon}}}W\right\|_{\infty}<\eta, \|Y^{\varepsilon}_{t}-Y^{\varepsilon}_{t_{i}^{n}}\|_{\beta}\leq \mu\right)
\leq \exp\left(-2R \log \log \frac{1}{\varepsilon}\right).
\end{equation}
\end{lemma}

\begin{proof}
For the simplicity of notation, we let,
\begin{equation*}
\Delta G_{s}^{\varepsilon}(v(x),w(x)):= G_{s}^{\varepsilon}(a,x,v(x))-G_{s}^{\varepsilon}(a,x,w(x)).
\end{equation*}

By the uniqueness of solutions of $S_{t}(h,y)$, we use its mild form and obtain,
\begin{eqnarray*}
\|Y_{t}^{\varepsilon}-S_{t}(h,y)\|_{\beta} &\leq& \sup_{y\in \mathbb{R}}\frac{e^{-\beta |y|}}{\sqrt{2\log \log \frac{1}{\varepsilon}}}\int_{0}^{t}\int_{U} P_{t-s}G_{s}^{\varepsilon}(a,y,Y_{s}^{\varepsilon}(y))W(dads) \\
&&+ \sup_{y\in \mathbb{R}} e^{-\beta |y|} \int_{0}^{t} \int_{U}P_{t-s} \Delta G_{s}^{\varepsilon}(Y_{s}^{\varepsilon}(y),0)h_{s}(a)\lambda(da)ds\\
&=& \frac{B_{1}^{\varepsilon}(t)}{\sqrt{2\log\log\frac{1}{\varepsilon}}} + B_{2}^{\varepsilon}(t).
\end{eqnarray*}
Thus, we may write,
\begin{eqnarray*}
&&P\left(\|Y_{t}^{\varepsilon}-S_{t}(h,y)\|_{\beta} >\rho, \left\|\frac{1}{\sqrt{2\log \log \frac{1}{\varepsilon}}} W\right\|_{\infty} <\eta, \|Y_{t}^{\varepsilon}-Y^{\varepsilon}_{t_{i}^{n}}\|_{\beta} \leq \mu\right)\\
&\leq& P\left(B_{1}^{\varepsilon}(t) >\frac{\rho \sqrt{2\log \log \frac{1}{\varepsilon}}}{2}, \left\|\frac{1}{\sqrt{2\log \log \frac{1}{\varepsilon}}} W\right\|_{\infty} <\eta\right) \\
&&+ P\left(B_{2}^{\varepsilon}(t)> \frac{\rho}{2}, \|Y_{t}^{\varepsilon}-Y^{\varepsilon}_{t_{i}^{n}}\|_{\beta} \leq \mu\right)\\
&=& I_{1}+I_{2}.
\end{eqnarray*}
By the same reasoning as in the proof of Lemma 3, there exists a constant $C>0$ such that $\mathbb{E}(\exp(B_{1}^{\varepsilon}(t)^{2}))\leq C$, which yields,
\begin{eqnarray}\label{B1}
I_{1} &=& P\left(\exp(B_{1}^{\varepsilon}(t)^{2}) > \exp\left(\frac{\rho^{2}(2\log \log \frac{1}{\varepsilon})}{4}\right),  \left\|\frac{1}{\sqrt{2\log \log \frac{1}{\varepsilon}}} W\right\|_{\infty} <\eta\right)\nonumber\\
&\leq& C \exp\left(\frac{-\rho^{2}(2\log \log \frac{1}{\varepsilon})}{4}\right)\leq \exp(-2R\log \log \frac{1}{\varepsilon}),
\end{eqnarray}
since the first inequality in \eqref{B1} is true for any $\rho>0$. As for $I_{2}$, noting that $t\in \Delta_{i}^{n}$, we have,
\begin{eqnarray*}
B_{2}^{\varepsilon}(t)&=& \sup_{y\in \mathbb{R}} e^{-\beta |y|}\int_{0}^{t_{i}^{n}}\int_{U} P_{t_{i}^{n}-s}G_{s}^{\varepsilon}(a,y,Y_{s}^{\varepsilon}(y))h_{s}(a)\lambda(da)ds\\
&&+ \sup_{y}e^{-\beta |y|}\int_{t_{i}^{n}}^{t} \int_{U} P_{t-s}G_{s}^{\varepsilon}(a,y,Y_{s}^{\varepsilon}(y))h_{s}(a)\lambda(da)ds\\
&&- \sup_{y}e^{-\beta |y|} \int_{0}^{t_{i}^{n}} P_{t_{i}^{n}-s} G_{s}^{\varepsilon}(a,y,0)h_{s}(a)\lambda(da)ds\\
&&- \sup_{y}e^{-\beta |y|}\int_{t_{i}^{n}}^{t} P_{t-s}G_{s}^{\varepsilon}(a,y,0)h_{s}(a)\lambda(da)ds,
\end{eqnarray*}
which leads to,
\begin{eqnarray*}
\sup_{t\in \Delta_{i}^{n}}B_{2}^{\varepsilon}(t)&=& \sup_{t\in \Delta_{i}^{n}}\sup_{y}e^{-\beta |y|} \int_{0}^{t_{i}^{n}} \int_{U} P_{t_{i}^{n}-s} \Delta G_{s}^{\varepsilon}(Y_{s}^{\varepsilon}(y),Y_{s_{i}^{n}}^{\varepsilon}(y))h_{s}(a)\lambda(da)ds\\
&&+ \sup_{t\in \Delta_{i}^{n}}\sup_{y} e^{-\beta |y|} \int_{0}^{t_{i}^{n}} \int_{U} P_{t_{i}^{n}-s} \Delta G_{s}^{\varepsilon}(Y_{s_{i}^{n}}^{\varepsilon}(y),0)h_{s}(a)\lambda(da)ds\\
&&+\sup_{t\in \Delta_{i}^{n}} \sup_{y} e^{-\beta |y|} \int_{t_{i}^{n}}^{t} \int_{U} P_{t-s}\Delta G_{s}^{\varepsilon}(Y_{s}^{\varepsilon}(y),0)h_{s}(a)\lambda(da)ds\\
&=& \sup_{t\in \Delta_{i}^{n}}I_{21}+\sup_{t\in \Delta_{i}^{n}}I_{22}+\sup_{t\in \Delta_{i}^{n}}I_{23}.
\end{eqnarray*}
Using \eqref{con1} and condition $\|Y_{s}^{\varepsilon}(y)-Y_{s_{i}^{n}}^{\varepsilon}(y)\|_{\beta} <\mu$, we have $\mathbb{E}\exp(|I_{21}|^{2})\leq \mu^{2}\tilde{k}_{1}|t_{i}^{n}|^{1/2}$ and we may bound $\mathbb{E}\exp(|I_{22}|^{2})$ and $\mathbb{E}\exp(|I_{23}|^{2})$ by $\tilde{k}_{2}|t_{i}^{n}|^{1/2}$ and $\tilde{k}_{3}|t-t_{i}^{n}|^{1/2}$, respectively. Then for any fixed $R>0$ and $\varepsilon \in (0,1)$ using the fact that $t_{i}^{n}:= (Ti)/n$, we may choose an $n_{0}\in \mathbb{N}$ such that for any $n\geq n_{0}$ the following inequality holds.
\begin{eqnarray*}
I_{2}&\leq& P\left(\sup_{t\in \Delta_{i}^{n}}I_{21}>\frac{\rho}{6}, \left\|Y_{t}^{\varepsilon}-Y^{\varepsilon}_{t_{i}^{n}}\right\|_{\beta} \leq \mu\right)
+P\left(\sup_{t\in \Delta_{i}^{n}}I_{22}>\frac{\rho}{6}, \left\|Y_{t}^{\varepsilon}-Y^{\varepsilon}_{t_{i}^{n}}\right\|_{\beta} \leq \mu\right)\\
&&+P\left(\sup_{t\in \Delta_{i}^{n}}I_{23}>\frac{\rho}{6}, \left\|Y_{t}^{\varepsilon}-Y^{\varepsilon}_{t_{i}^{n}}\right\|_{\beta} \leq \mu\right)\\
&\leq& \exp(-2R\log \log \frac{1}{\varepsilon}).
\end{eqnarray*}
Hence, we obtain \eqref{2ndequation}.

\end{proof}

Notice that the above estimates hold with any $a(\varepsilon)$ instead of $1/\sqrt{2\log \log (1/\varepsilon)}$ satisfying \eqref{conditions} and thus we achieve the MDP for the class of SPDEs by Azencott method, where using the rate function from Schilder's theorem and letting $\Phi(h)=S_{t}(h,y)$, we obtain \eqref{rate1}. For the Strassen's compact LIL to prove the relative compactness of $Z_{t}^{\varepsilon}(y)$, we show its tightness property by following the well established theorem stated below.

\begin{theorem}[Theorem $12.3$ in \cite{Billingsley}]\label{T12.3}
The sequence $\{X^{\varepsilon}\}_{\varepsilon>0}$ is tight in $\mathcal{C}\left([0,1];\mathbb{R}\right)$, if\\
$(i)$ the sequence $\{X^{\varepsilon}(0)\}_{\varepsilon>0}$ is tight,\\
$(ii)$ there exist constants $\gamma \geq 0$ and $\alpha >1$ and a nondecreasing, continuous function $F$ on $[0,1]$ such that \begin{equation}\label{12.3}
P\left(\left|X^{\varepsilon}(t_{2})-X^{\varepsilon}(t_{1})\right|\geq \lambda\right) \leq \frac{1}{\lambda^{\gamma}}\left|F(t_{2})-F(t_{1})\right|^{\alpha},
\end{equation}
holds for all $t_{1},t_{2}$ and $n$ and all positive $\lambda$.
\end{theorem}
We need the following analogous result to Lemma 2 for the process $\{Z^{\varepsilon}_{t}(y)\}_{0<\varepsilon<1}$, where its proof is omitted due to its similarity with the proof of Lemma 1 in \cite{moderate}.

\begin{lemma}\label{Mlemma}
Let $Z_{t}^{\varepsilon}(y)$ be the unique solution to SPDE \eqref{SPDE}, then for any $p\geq 1, 0<\varepsilon<1$ and $T>0$, there exists a positive constant $M_{2}$ such that,
\begin{equation}\label{Mequation}
\sup_{\varepsilon >0}\mathbb{E}\left(\sup_{0\leq t \leq T} \int_{\mathbb{R}} Z_{t}^{\varepsilon}(x)^{2}e^{-2\beta|x|}dx\right)^{p}\leq M_{2}.
\end{equation}
\end{lemma}

\begin{theorem}\label{tight}
 Family, $\{Z_{t}^{\varepsilon}\}_{0<\varepsilon<1}$ takes values in $\mathcal{C}\left([0,1];\mathbb{B}_{\beta}\right)$ and is tight.
\end{theorem}

\begin{proof}
 It was shown in Lemma 3 of \cite{moderate} that $v_{t}^{\varepsilon}(y)$ defined by \eqref{centered} takes values in $\mathcal{C}\left([0,1];\mathbb{B}_{\beta}\right)$ and its solution is unique allowing us to use the mild solution,
 \begin{equation}
Z_{t}^{\varepsilon}(y):= \frac{1}{\sqrt{2\log \log \frac{1}{\varepsilon}}}\int_{0}^{t}\int_{U}P_{t-s}G_{s}^{\varepsilon}(a,y,Z_{s}^{\varepsilon}(y))W(dads),
\end{equation}
where $P_{t-s}$ is the Brownian semigroup. Let $0<\varepsilon<1$ and $t_{1},t_{2}\in [0,1]$ be arbitrary with $t_{1}<t_{2}$. For $n>8$, we proceed as follows,
\begin{eqnarray*}
&&\mathbb{E}\left\|Z_{t_{2}}^{\varepsilon}(x)-Z_{t_{1}}^{\varepsilon}(x)\right\|_{\beta}^{n}\\
&=& \mathbb{E}\sup_{x\in \mathbb{R}}e^{-n\beta |x|}\left|\frac{1}{\sqrt{2\log \log \frac{1}{\varepsilon}}} \int_{0}^{t_{2}}\int_{U}\int_{\mathbb{R}}
\frac{1}{\sqrt{2\pi \left(t_{2}-s\right)}} e^{-\frac{|x-y|^{2}}{2\left(t_{2}-s\right)}}G_{s}^{\varepsilon}\left(a,y,Z_{s}^{\varepsilon}(y)\right)dyW(dads)\right.\\
&&\left.- \frac{1}{\sqrt{2\log \log \frac{1}{\varepsilon}}} \int_{0}^{t_{1}} \int_{U}\int_{\mathbb{R}} \frac{1}{\sqrt{2\pi \left(t_{1}-s\right)}} e^{-\frac{|x-y|^{2}}{2\left(t_{1}-s\right)}} G_{s}^{\varepsilon}\left(a,y,Z_{s}^{\varepsilon}(y)\right)dyW(dads)\right|^{n}.
\end{eqnarray*}
For better presentation, let $K_{0}:= \frac{1}{\sqrt{2 \log \log \frac{1}{\varepsilon}} \sqrt{2\pi}}$ and $\tilde{G}_{s}^{\varepsilon}(a,y):= G_{s}^{\varepsilon}\left(a,y,Z_{s}^{\varepsilon}(y)\right)$. We will call the first integral above $I(t_{2},t_{2})(x)$, where the first $t_{2}$ appears in the upper limit of the integral and the second is the time parameter in the Gaussian density. Similarly, the second integral is denoted as $I\left(t_{1},t_{1}\right)(x)$. Using this notation we have,
\begin{eqnarray*}
&&\mathbb{E}\left\|Z_{t_{2}}^{\varepsilon}(x)-Z_{t_{1}}^{\varepsilon}(x)\right\|_{\beta}^{n}
=\mathbb{E}\left\|I\left(t_{2},t_{2}\right)(x)-I\left(t_{1},t_{1}\right)(x)\right\|_{\beta}^{n}\\
&\leq&  2^{n-1}K_{0}^{n}\left[\mathbb{E}\left\|I\left(t_{2},t_{2}\right)(x)-I\left(t_{1},t_{2}\right)(x)\right\|_{\beta}^{n}+ \mathbb{E}\left\|I\left(t_{1},t_{2}\right)(x)-I\left(t_{1},t_{1}\right)(x)\right\|_{\beta}^{n}\right]\\
&\leq& 2^{n-1}\left(J_{1}+J_{2}\right).
\end{eqnarray*}
As for $J_{1}$, applying the Burkholder-Davis-Gundy inequality yields,
\begin{equation*}
J_{1}\leq K_{0}^{n}\mathbb{E}\sup_{x\in \mathbb{R}}e^{-n\beta |x|}\left|\int_{t_{1}}^{t_{2}} \int_{U}\left(\int_{\mathbb{R}} \frac{1}{\sqrt{t_{2}-s}}e^{-\frac{|x-y|^{2}}{2\left(t_{2}-s\right)}} \tilde{G}_{s}^{\varepsilon}(a,y)dy\right)^{2}\lambda(da)ds\right|^{\frac{n}{2}},
\end{equation*}
where by H\"older's inequality and condition \eqref{condition2},
\begin{eqnarray}\label{2}
&&\int_{U}\left(\int_{\mathbb{R}} \frac{1}{\sqrt{t_{2}-s}}e^{-\frac{|x-y|^{2}}{2\left(t_{2}-s\right)}}\tilde{G}_{s}^{\varepsilon}(a,y)dy\right)^{2}\lambda(da)\\
&\leq& \int_{U}\int_{\mathbb{R}} \frac{1}{t_{2}-s} e^{-\frac{|x-y|^{2}}{t_{2}-s}}e^{2\beta |y|}dy
\int_{\mathbb{R}} \tilde{G}_{s}^{\varepsilon}(a,y)^{2}e^{-2\beta|y|}dy\lambda(da)\nonumber\\
&\leq& K_{4}\int_{\mathbb{R}}\frac{1}{t_{2}-s} e^{-\frac{|x-y|^{2}}{t_{2}-s}} e^{2\beta |y|}dy \int_{\mathbb{R}}
\left(1+\left(2\varepsilon \log \log \frac{1}{\varepsilon}\right)Z_{s}^{\varepsilon}(y)^{2} + e^{2\beta_{0}|y|}\right) e^{-2\beta |y|}dy,\nonumber
\end{eqnarray}
and by our assumption, $\beta_{0}<\beta$ and $\beta>0$. Moreover, similar to \eqref{pinequality},
\begin{eqnarray*}
 \int_{t_{1}}^{t_{2}}\int_{\mathbb{R}}\frac{1}{t_{2}-s}e^{-\frac{|x-y|^{2}}{t_{2}-s}}e^{2\beta |y|}dyds
\leq k_{1}\int_{t_{1}}^{t_{2}}\int_{\mathbb{R}} p_{t_{2}-s}^{2}(x-y)e^{2\beta |y|}dy ds \leq k_{2}e^{2\beta |x|}\sqrt{t_{2}-t_{1}},
\end{eqnarray*}
therefore, using Lemma \ref{Mlemma},
\begin{eqnarray}\label{J1bound}
J_{1}&\leq & K_{0}^{n}K_{4}k_{3} \mathbb{E}\left|\sqrt{t_{2}-t_{1}}\int_{\mathbb{R}}
\left(2\varepsilon \log \log \frac{1}{\varepsilon}\right)\sup_{t_{1}\leq s\leq t_{2}}Z_{s}^{\varepsilon}(y)^{2} e^{-2\beta |y|}dy\right|^{\frac{n}{2}}\nonumber\\
&\leq& k_{4}M_{2}\left|t_{2}-t_{1}\right|^{\frac{n}{4}},
\end{eqnarray}
where $k_{4}M_{2}$ is independent of $\varepsilon$. Using notation \eqref{deltap}, we continue by estimating $J_{2}$,
\begin{eqnarray*}
J_{2}&=& K_{0}\mathbb{E}\left\|\int_{0}^{t_{1}}\int_{U}\int_{\mathbb{R}} \Delta p\left(t_{2},t_{1}\right)G_{s}^{\varepsilon}(a,y)dy W(dads)\right\|_{\beta}^{n}\\
&\leq& K_{0}^{n}\mathbb{E}\sup_{x\in \mathbb{R}}e^{-n\beta |x|}\left|\int_{0}^{t_{1}}\int_{U}\int_{\mathbb{R}}\left(\Delta p\left(t_{2},t_{1}\right)\right)^{2}e^{2\beta |y|}dy
\int_{\mathbb{R}} \tilde{G}_{s}^{\varepsilon}(a,y)^{2}e^{-2\beta |y|}dy\lambda(da)ds\right|^{\frac{n}{2}}\\
&\leq& k_{5}\sup_{x\in \mathbb{R}}e^{-n\beta |x|}\left|\int_{0}^{t_{1}} \int_{\mathbb{R}} \left(\Delta p\left(t_{2},t_{1}\right)\right)^{2}
e^{2\beta |y|}dyds \right|^{\frac{n}{2}},
\end{eqnarray*}
where steps similar to those taken for estimating $J_{1}$ were applied. It can be seen that for $0<\alpha\leq 1/2$,
\begin{eqnarray}\label{Deltap}
\left(\Delta p\left(t_{2},t_{1}\right)\right)^{2}&=&\left|p_{t_{2}-s}(x-y)-p_{t_{1}-s}(x-y)\right|^{\alpha}
\left|p_{t_{2}-s}(x-y)-p_{t_{1}-s}(x-y)\right|^{2-\alpha}\\
&\leq& 2^{1-\alpha} \left|p_{t_{2}-s}(x-y)-p_{t_{1}-s}(x-y)\right|^{\alpha}\left(p_{t_{2}-s}(x-y)^{2-\alpha}+
p_{t_{1}-s}(x-y)^{2-\alpha}\right).\nonumber
\end{eqnarray}
Also $\Delta p\left(t_{2},t_{1}\right)$ may be written as,
\begin{equation}\label{3}
 \frac{1}{\sqrt{2\pi}} \frac{1}{\sqrt{t_{2}-s}} \left|e^{-\frac{|x-y|^{2}}{2\left(t_{2}-s\right)}}
-e^{-\frac{|x-y|^{2}}{2\left(t_{1}-s\right)}}\right|+\frac{1}{\sqrt{2\pi}}
e^{-\frac{|x-y|^{2}}{2\left(t_{1}-s\right)}}\left|\frac{1}{\sqrt{t_{2}-s}}-\frac{1}{\sqrt{t_{1}-s}}\right|=: I_{1}+I_{2}.
\end{equation}
Using this form in \eqref{Deltap} we obtain,
\begin{equation*}
\left(\Delta p\left(t_{2},t_{1}\right)\right)^{2}\leq
K\left|I_{1}+I_{2}\right|^{\alpha}\left(p_{t_{2}-s}(x-y)^{2-\alpha}+p_{t_{1}-s}(x-y)^{2-\alpha}\right),
\end{equation*}
hence,
\begin{eqnarray*}
J_{2}&\leq& k_{5}\sup_{x\in \mathbb{R}}e^{-n\beta |x|}\left|\int_{0}^{t_{1}}\int_{\mathbb{R}} \left|I_{1}\right|^{\alpha}p_{t_{2}-s}(x-y)^{2-\alpha}e^{2\beta |y|}dyds+ \int_{0}^{t_{1}}\int_{\mathbb{R}}\left|I_{2}\right|^{\alpha}p_{t_{2}-s}(x-y)^{2-\alpha}e^{2\beta |y|}dyds\right.\\
&&\left.+ \int_{0}^{t_{1}}\int_{\mathbb{R}} \left|I_{1}\right|^{\alpha}p_{t_{1}-s}(x-y)^{2-\alpha}e^{2\beta |y|}dyds+\int_{0}^{t_{1}}\int_{\mathbb{R}} \left|I_{2}\right|^{\alpha}p_{t_{1}-s}(x-y)^{2-\alpha}e^{2\beta |y|}dyds\right|^{\frac{n}{2}}\\
&=& k_{5}\sup_{x\in \mathbb{R}} e^{-n\beta |x|}\left|J_{2,1}+ J_{2,2}+ J_{2,3}+ J_{2,4}\right|^{\frac{n}{2}}.
\end{eqnarray*}
 For $I_{1}$, we use the mean value theorem to obtain,
\begin{equation}\label{4}
I_{1}\leq k_{6}\frac{|x-y|^{2}}{2\sqrt{2\pi(t_{2}-s)}}\left|\frac{1}{t_{2}-s}-\frac{1}{t_{1}-s}\right|
=k_{6}\frac{|x-y|^{2}}{2\sqrt{2\pi(t_{2}-s)}}\frac{\left|t_{2}-t_{1}\right|}{\left(t_{2}-s\right)\left(t_{1}-s\right)}.
\end{equation}

In particular,
\begin{eqnarray*}
J_{2,1}&\leq& k_{6}\int_{0}^{t_{1}}\int_{\mathbb{R}} \frac{|x-y|^{2\alpha}}{2^{\alpha}\left(2\pi \left(t_{2}-s\right)\right)^{\frac{\alpha}{2}}}
 \frac{\left|t_{2}-t_{1}\right|^{\alpha}}{\left(t_{2}-s\right)^{\alpha}\left(t_{1}-s\right)^{\alpha}}
p_{t_{2}-s}(x-y)^{2-\alpha}e^{2\beta |y|}dyds\\
&\leq&k_{7} \int_{0}^{t_{1}}\int_{\mathbb{R}} \frac{\left|t_{2}-t_{1}\right|^{\alpha}}{\left(t_{2}-s\right)
^{\frac{3\alpha}{2}}\left(t_{1}-s\right)^{\alpha}}|x-y|^{2\alpha} p_{t_{2}-s}(x-y)^{2-\alpha}e^{2\beta|y|}dyds\\
&\leq&k_{7}\int_{0}^{t}\int_{\mathbb{R}}\frac{\left|t_{2}-t_{1}\right|^{\alpha}}{\left(t_{2}-s\right)
^{\frac{3\alpha}{2}}\left(t_{1}-s\right)^{\alpha}} \frac{|x-y|^{2\alpha}}{\left(t_{2}-s\right)^{1-\frac{\alpha}{2}}}\sqrt{\frac{t_{2}-s}{2-\alpha}}
p_{\frac{t_{2}-s}{2-\alpha}}(x-y)e^{2\beta |y|}dyds\\
&\leq& k_{7}\int_{0}^{t_{1}}\int_{\mathbb{R}}\frac{\left|t_{2}-t_{1}\right|^{\alpha}}
{\left(t_{2}-s\right)^{\frac{1}{2}+\alpha}\left(t_{1}-s\right)^{\alpha}}|x-y|^{2\alpha}
p_{\frac{t_{2}-s}{2-\alpha}}(x-y)e^{2\beta |y|}dyds\\
&\leq& k_{8} e^{2\beta |x|}\int_{0}^{t_{1}} \frac{\left|t_{2}-t_{1}\right|^{\alpha}}{\left(t_{2}-s\right)^{\frac{1}{2}+\alpha}
\left(t_{1}-s\right)^{\alpha}}ds.
\end{eqnarray*}
Thus, noting the assumption $t_{1}<t_{2}$, we arrive at,
\begin{equation*}
J_{2,1} \leq k_{8}e^{2\beta|x|}\left|t_{2}-t_{1}\right|^{\alpha}\int_{0}^{t_{1}}
\left(t_{1}-s\right)^{-\left(\frac{1}{2}+2\alpha\right)}ds\leq k_{9}e^{2\beta |x|}\left|t_{2}-t_{1}\right|^{\alpha},
\end{equation*}
if $2\alpha < 1/2$. Similarly for $J_{2,3}$,
\begin{equation*}
J_{2,3}\leq k_{10}e^{2\beta |x|}\left|t_{2}-t_{1}\right|^{\alpha},
\end{equation*}
if $2\alpha < 1/2$. To determine bounds for $J_{2,2}$ and $J_{2,4}$, we have for $i,j=1,2$ with $i\neq j$,
\begin{eqnarray*}
&&\int_{\mathbb{R}} \left|\frac{1}{\sqrt{t_{1}-s}}-\frac{1}{\sqrt{t_{2}-s}}\right|^{\alpha}
p_{t_{i}-s}(x-y)^{2-\alpha}e^{2\beta |y|}dy\\
&\leq& k_{11}\int_{\mathbb{R}} \left|\frac{t_{2}-t_{1}}{\left(t_{1}-s\right)\sqrt{t_{2}-s}+\left(t_{2}-s\right)\sqrt{t_{1}-s}}\right|^{\alpha}
\frac{1}{\left(\sqrt{t_{i}-s}\right)^{1-\alpha}}p_{\frac{t_{i}-s}{2-\alpha}}(x-y)e^{2\beta |y|}dy\\
&\leq& k_{11} \frac{\left|t_{2}-t_{1}\right|^{\alpha}}{\left(t_{j}-s\right)^{\alpha}}\frac{e^{2\beta |x|}}{\left(t_{i}-s\right)^{\frac{1}{2}}},
\end{eqnarray*}
and
\begin{equation*}
k_{11}\int_{0}^{t_{1}}\frac{\left|t_{2}-t_{1}\right|^{\alpha}}{\left(t_{j}-s\right)^{\alpha}}
\frac{e^{2\beta |x|}}{\left(t_{i}-s\right)^{\frac{1}{2}}}ds
\leq k_{11}e^{2\beta |x|} \left|t_{2}-t_{1}\right|^{\alpha}\int_{0}^{t_{1}}\left(t_{1}-s\right)^{-\left(\alpha + \frac{1}{2}\right)}ds
\leq k_{11}e^{2\beta |x|} \left|t_{2}-t_{1}\right|^{\alpha},
\end{equation*}
for $\alpha < 1/2$. From values for $\alpha$ found above for each term of $J_2$, we require $0<\alpha < 1/4$ and obtain,
\begin{equation*}
J_{2}\leq k_{12}\left|t_{2}-t_{1}\right|^{\frac{\alpha n}{2}},
\end{equation*}
where $k_{12}$ is independent of $\varepsilon$. Furthermore, noting the bound for $J_{1}$ in \eqref{J1bound} we confirm our assumption of $n>8$ required to satisfy condition \eqref{12.3}.
 \end{proof}
We now verify that the limit set for $\{Z^{\varepsilon}_{t}(y)\}_{0<\varepsilon<1}$ is $L_{1}$ given in Theorem 4. For better presentation, we let $\varepsilon = 1/(c^{j})$, where $c>1$ and $j\geq 1$.
 \begin{lemma}\label{limit}
For any $g\in L_{1}$, $\varepsilon >0$, and $c>1$, there exists $j_{0}\in \mathbb{N}$ such that for every $j>j_{0}$, $P\left(\|Z^{\frac{1}{c^{j}}}-g\|_{\beta}\leq \varepsilon \text{ i.o.}\right)=1$.
\end{lemma}
\begin{proof}
Let $g\in L_{1}$ and $h_{s}\in L^{2}\left([0,1]\times U, ds\lambda(da)\right)$ such that $g=S_{t}(h_{s},y)$ and \\ $\frac{1}{2}\int_{0}^{t}\int_{U}|h_{s}(a)|^{2}\lambda(da)ds \leq 1$. Denote,
\begin{equation*}
F_{j}:=\left\{\|Z^{\frac{1}{c^{j}}}-g\|_{\beta}\leq \varepsilon\right\} \text{ and } G{j}:= \left\{\left\|\frac{1}{\sqrt{2\log \log c^{j}}}W-h\right\|_{\infty}\leq \eta \right\},
\end{equation*}
for some constant $\eta >0$. We need to prove that $P\left(\limsup_{j\rightarrow \infty}F_{j}\right)=1$. Based on the Strassen's compact LIL for Brownian sheets proved in Section 1.4 in \cite{Stroock}, $P\left(\limsup_{j}G_{j}\right)=1$. Let $R>1$, then by \eqref{maininequality} we have,
\begin{equation}\label{Fj}
P\left(F_{j}^{c}\cap G_{j}\right)\leq \exp\left(-2R\log \log c^{j}\right)= \frac{K_{R}}{j^{2R}}\leq \frac{K_{R}}{j^{2}},
\end{equation}
where we used the fact that for $k\in \mathbb{R}$,
\begin{equation*}
\exp\left(-k\log\log c^{j}\right)= \frac{K_{k}}{j^{k}}.
\end{equation*}
Now by the Borel-Cantelli lemma applied to (\ref{Fj}), we arrive at,
\begin{equation*}
P\left(\limsup_{j\rightarrow \infty}F_{j}^{c}\cap G_{j}\right)=0.
\end{equation*}
Thus, we obtain,
\begin{equation*}
1= P\left(\limsup_{j\rightarrow \infty }G_{j}\right)\leq P\left(\limsup_{j \rightarrow \infty}G_{j}\cap F_{j}\right)+ P\left(\limsup_{j\rightarrow \infty}G_{j}\cap F_{j}^{c}\right)\leq P\left(\limsup_{j\rightarrow \infty}F_{j}\right),
\end{equation*}
obtaining the result.
\end{proof}

As for the Classical LIL, note that it is sufficient to prove \eqref{limsup}, where \eqref{liminf} may be proved analogously. As above, we use the notation, $\varepsilon:=1/(c^{j})$. Then for every $\varepsilon >0$, \eqref{limsup} is equivalent to
\begin{equation*}
P\left(\left\| \frac{u_{t}^{\frac{1}{c^{j}}}(y)-u^{0}_{t}(y)}{\sqrt{\frac{2}{c^{j}}\log \log c^{j}}} - 1 \right\|_{\beta} >\varepsilon \text{  i.o.}\right)=0,
\end{equation*}
which may further be written as,
\begin{equation}\label{frac}
\lim_{j\rightarrow \infty} P\left(\sup_{k\geq j} \left\| \frac{u^{\frac{1}{c^{k}}}_{t}(y)-u^{0}_{t}(y)}{\sqrt{\frac{2}{c^{k}}\log \log c^{k}}}-1\right\|_{\beta}>\varepsilon\right)=0.
\end{equation}
By Chebyshev inequality,
\begin{eqnarray*}
&&\lim_{j\rightarrow \infty} P\left( \sup_{k\geq j}\sup_{y}e^{-\beta |y|} \frac{|u_{t}^{\frac{1}{c^{k}}}(y)-u^{0}_{t}(y)|}{\sqrt{\frac{2}{c^{k}}\log\log c^{k}}}>\varepsilon + \sup_{y\in \mathbb{R}}e^{-\beta |y|}\right) \\
&\leq& (\varepsilon + \sup_{y}e^{-\beta |y|})^{-2}\lim_{j\rightarrow \infty} \mathbb{E} \sup_{k\geq j} \sup_{y}e^{-2\beta |y|}
\frac{|u_{t}^{\frac{1}{c^{k}}}(y)-u^{0}_{t}(y)|^{2}}{\frac{2}{c^{k}}\log \log c^{k}}.
\end{eqnarray*}
Furthermore,
\begin{eqnarray}\label{righthand}
&&\mathbb{E} \sup_{k\geq j} \sup_{y} e^{-2\beta |y|} \left|u_{t}^{\frac{1}{c^{k}}}(y)-u^{0}_{t}(y)\right|^{2}\nonumber \\
&\leq&  \mathbb{E} \sup_{k\geq j} \sup_{y}e^{-2\beta |y|}\frac{1}{c^{k}}\left|\int_{0}^{t}\int_{U}\int_{\mathbb{R}}p_{t-s}(x-y)G(a,x,u_{s}^{\varepsilon}(x))dxW(dads)\right|^{2},
\end{eqnarray}
which analogous to previous estimates in this section may be bounded above by $(1/c^{k})\tilde{M}$ for a positive constant, $\tilde{M}$. Hence, we arrive at,
\begin{eqnarray*}
&&\lim_{j\rightarrow \infty} P\left(\sup_{k\geq j} \sup_{y\in \mathbb{R}} e^{-\beta |y|} \frac{|u_{t}^{\frac{1}{c^{k}}(y)}-u^{0}_{t}(y)|}{\sqrt{\frac{2}{c^{k}}\log \log c^{k}}} >\varepsilon + \sup_{y} e^{-\beta |y|}\right)\\
&\leq& (\varepsilon + \sup_{y}e^{-\beta |y|})^{-2}\lim_{j\rightarrow \infty}\sup_{k\geq j}\frac{\tilde{M}}{(2\log \log c^{k})},
\end{eqnarray*}
which implies \eqref{frac}, noting that $\varepsilon >0$ was arbitrarily chosen.

\section{LIL for SBM and FVP}
In this section we apply the results from Section Three to derive the MDP and the two types of LIL for SBM and FVP. To achieve the MDP for these models, using relations \eqref{SBM d} and \eqref{FVP d}, we have by \eqref{w}, $v_{t}^{\varepsilon}(y)= \int_{0}^{y}\omega_{t}^{\varepsilon}(dx)$ for SBM and $v_{t}^{\varepsilon}(y)= \int_{-\infty}^{y}\omega_{t}^{\varepsilon}(dx)$ for FVP. In Lemma 6 in \cite{large} it was shown that for $\mathcal{A}$, the set of nondecreasing functions, the map, $\xi: \mathbb{B}_{\beta} \cap \mathcal{A} \rightarrow \mathcal{M}_{\beta}(\mathbb{R})$ defined as,
\begin{equation}\label{int}
\xi(u)(B)= \int 1_{B}(y)du(y),
\end{equation}
is continuous for all $B\in \mathbb{B}(\mathbb{R})$. Also the same reasoning may be applied to prove that the map $\tilde{\xi}:\mathbb{B}_{\beta} \cap \mathcal{A}\rightarrow \mathcal{P}_{\beta}(\mathbb{R})$ defined by \eqref{int} is continuous. Hence, SBM and FVP may be written as continuous functions of the solution of SPDE, $v_{t}^{\varepsilon}(y)$, and thus by the contraction principle, MDP follows for SBM and FVP. It is left to determine their exact form of rate function. In Section 4 of \cite{moderate}, for $h_{s}\in L^{2}\left([0,1]\times U, ds\lambda(da)\right)$, the following expression was derived for $\frac{1}{2}\inf \int_{0}^{1}\int_{U}\left|h_{s}(a)\right|^{2}\lambda(da)ds$ for both SBM and FVP by letting $a=u_{t}^{0}(y)$,
\begin{equation*}
\frac{1}{2}\int_{0}^{1}\int_{U}\left|h_{s}(a)\right|^{2}\lambda(da) ds= \frac{1}{2}\int_{0}^{1}\int_{\mathbb{R}}\left|\frac{d\left(\dot{\omega}-\frac{1}{2}\Delta^{*}\omega\right)}
{d\mu_{t}^{0}}y\right|\mu_{t}^{0}(dy)dt,
\end{equation*}
where $\mu_{t}^{0}\in \mathcal{H}_{\omega_{0}}$ in the case of SBM and $\mu_{t}^{0} \in \tilde{\mathcal{H}}_{\omega_{0}}$ for FVP and hence we obtain \eqref{rate2} and \eqref{rate3}. \\
For the two types of LIL, as stated in Section Two, based on the SPDE characterization of the two population models, SBM and FVP are given by the class of SPDE (\ref{SPDE}) with $G(a,y,u_{s}^{\varepsilon}(y)):=1_{0<a<u_{s}^{\varepsilon}(y)}+1_{u_{s}^{\varepsilon}(y)<a<0}$ and $G(a,y,u_{s}^{\varepsilon}(y):= 1_{a\leq u_{s}^{\varepsilon}(y)}-u_{s}^{\varepsilon}(y)$, respectively, where in both cases, $G(a,y,u_{s}^{\varepsilon}(y))$ satisfies conditions \eqref{con1} and \eqref{con2}. Using these SPDE characterizations we may form the analog process $Z_{t}^{\varepsilon}(y)$ and by estimates in Section Three obtain the Freidlin-Wentzell inequality \eqref{maininequality} and tightness of $\{Z_{t}^{\varepsilon}(y)\}_{0<\varepsilon<1}$ in $\mathcal{C}([0,1];\mathcal{M}_{\beta})$ and $\mathcal{C}([0,1];\mathcal{P}_{\beta})$, respectively. Furthermore, we have the continuity of their controlled PDE version $S_{t}(h,y)$ with respect to uniform topology when restricted on level sets $\{\tilde{I}\leq a\}$. Thus, we obtain the Strassen's compact LIL for SBM and FVP with limit sets being $L_{2}$ and $L_{3}$, respectively. The classical LIL may also be achieved in respective spaces following the same steps as in Section Three with no additional conditions required.


\begin{thebibliography}{49}

\bibitem{Azencott}
R. Azencott (1980). \emph{Grandes d$\acute{e}$viations et applications}. \emph{Ecole d'$\acute{e}$t$\acute{e}$ de Probabilit$\acute{e}$ de Saint-Flour VIII}. Lecture Notes in Mathematics, vol. 774, Springer, Berlin.

\bibitem{Baldi1}
P. Baldi (1986). Large deviations and functional iterated logarithm law for diffusion processes. \emph{Probab. Th. Rel. Fields.} vol. 71, no. 3, 435-453.

\bibitem{Baldi2}
P. Baldi (1991). Large deviations for diffusion processes with homogenization and applications. \emph{Ann. Probab.} vol. 19, no. 2, 509-524.

\bibitem{Billingsley}
P. Billingsley (1968). \emph{Convergence of Probability Measures}. Wiley Series in Probability and Mathematical Statistics.


\bibitem{Bingham}
N. Bingham (1986). Variants on the law of the iterated logarithm. \emph{Bull. Lond. Math. Soc.} vol. 18, no. 5, 433-467.

\bibitem{Budhiraja2}
A. Budhiraja, P. Dupuis (2000). A variational representation for positive functionals of infinite-dimensional Brownian motion. \emph{Probab. Math. Stat.} vol. 20, no. 1, 39-61.

\bibitem{Budhiraja1}
A. Budhiraja, P. Dupuis (2019). \emph{Analysis and Approximation of Rare Events: Representations and Weak Convergence Methods}. Probability Theory and Stochastic Modeling. vol. 94, Springer, Berlin.

\bibitem{BDM}
A. Budhiraja, P. Dupuis and V. Maroulas (2008). Large Deviations for
Infinite Dimensional Stochastic Dynamical Systems. {\em Ann. Probab.}vol. 36, no. 4, 1390-1420.

\bibitem{Burgers}
C. Cardon-Weber (1999). Large deviations for a Burgers' type SPDE. \emph{Stoch. Proc. Appl.} vol. 84, no. 1, 53-70.

\bibitem{Cerrai}
S. Cerrai, Michael R\"ockner (2004). Large deviations for stochastic reaction-diffusion systems with multiplicative noise and non-Lipschitz reaction term. \emph{Ann. Probab.} vol. 32, no. 1B, 1100-1139.

\bibitem{Y.Chen}
Y. Chen and L. Liu (2015). Laws of the iterated logarithm and a moderate deviation of MLE for the proportional hazards model with incomplete information. \emph{Comm. Stat.} vol. 44, no. 22, 4696-4708.

\bibitem{Chover}
J. Chover (1966). A law of the iterated logarithm for stable summands. \emph{Proc. Amer. Math. Soc.} vol. 17, no. 2, 441-443.

\bibitem{Chung}
K. Chung (1948). On the maximum partial sums of sequences of independent random variables. \emph{Trans. Amer. Math. Soc.} vol. 64, no. 2, 205-233.

\bibitem{Davie}
G. Davie (2001). The Borel-Cantelli lemmas, probability laws and Kolmogorov complexity. \emph{Ann. Probab.} vol. 29, no. 4, 1426-1434.

\bibitem{Dawson}
D. Dawson (1993). \emph{Measure-valued Markov processes}. $\acute{E}$cole d$'\acute{E}$t$\acute{e}$ de
Probabilit$\acute{e}$s de Saint-Flour XXI-1991, 1-260,  Lecture Notes in
Math., 1541, Springer, Berlin.

\bibitem{DF1}
D. Dawson and S. Feng (1998). Large deviations for the
Fleming-Viot process with neutral mutation and selection. II. {\em
Stoch. Proc. Appl.} vol. 77, no. 2, 207-232.

\bibitem{Dembo2}
A. Dembo and T. Zajic (1997). Uniform large and moderate deviations for functional empirical processes. \emph{Stoch. Proc. Appl.} vol. 67, no. 2, 195-211.

\bibitem{Dembo}
A. Dembo and O. Zeitouni (2010). {\em Large Deviations Techniques
and Applications}. Springer, Berlin.

\bibitem{Stroock}
J. Deuschel and D. Stroock (1989). \emph{Large Deviations}. Academic Press, Inc.

\bibitem{Divanji}
G. Divanji and K. Vidyalaxmi (2011). A precise asymptotic behaviour of the large deviation probabilities for weighted sums. \emph{Appl. Math.} vol. 2, no. 9, 1175-1181.

\bibitem{Dupuis}
P. Dupuis and R. Ellis (1997). {\em A Weak Convergence Approach to
the Theory of Large Deviations}. Wiley Series in Probability and
Statistics.

\bibitem{Dynkin}
E.B. Dynkin (1994). \emph{An Introduction to Branching Measure-Valued Processes}. American Mathematical Society CRM Monograph Series, Vol. 6.

\bibitem{Eddhabi}
M. Eddahbi and M. N'ZI (2002). Strassen's local law for diffusion processes under strong topologies. \emph{Acta. Math. Vietnamica.} vol. 27, no. 2, 151-163.

\bibitem{Etheridge}
 A. Etheridge (2000). {\em An Introduction to Superprocesses.}
 University Lecture Series, vol. 20. American Mathematical Society.

 \bibitem{Kurtz1}
 S. Ethier and T. Kurtz (1993).Fleming-Viot processes in population genetics. \emph{SIAM J. Control Optim.} vol. 31, no. 2, 345-386.

 \bibitem{Kurtz2}
 S. Ethier and T. Kurtz (2005). \emph{Markov Processes: Characterization and Convergence}. Wiley Series in Prob. and Stat.

\bibitem{large}
P. Fatheddin and J. Xiong (2015). Large deviation principle for some measure-valued processes. \emph{Stoch. Proc. Appl.} vol. 125, no. 3, 970-993.

\bibitem{moderate}
P. Fatheddin and J. Xiong (2016). Moderate deviation principle for a class of SPDEs. \emph{J. Appl. Probab.} vol. 53, no. 1, 279-292.

\bibitem{FenX}
S. Feng and J. Xiong (2002). Large deviations and quasi-potential of
a Fleming-Viot process. (English summary) {\em Electron. Comm.
Probab. 7}, 13-25 (electronic).

\bibitem{Gao}
F. Gao and S. Wang (2011). Asymptotic behavior of the empirical conditional value-at-risk. \emph{Insurance Math. Econom.} vol. 49, no. 3, 345-352.

\bibitem{Hong1}
W. Hong (2002). Longtime behavior for the occupation time process of a super-Brownian motion with random immigration. \emph{Stoch. Proc. Appl.} vol. 102, no. 1, 43-62.

\bibitem{Jing}
B. Jing, Q. Shao and Q. Wang (2003). Self-normalized Cramer-type large deviations for independent random variables. \emph{Ann. Probab.} vol. 31, no. 4, 2167-2215.

\bibitem{Khintchine}
A. Khintchine (1924). $\ddot{U}$ber einen satz der wahrscheinlichkeitsrechnung. \emph{Fund. Math.} vol. 6, no. 1, 9-20.

\bibitem{Mogulskii}
A. Mogul'skii (1980). On the law of the iterated logarithm in Chung's form for functional spaces. \emph{Theo. Probab. Appl.} vol. 24, no. 2, 405-413.

\bibitem{Nualart}
D. Nualart and C. Rovira (2000). Large deviations for stochastic Volterra equations. \emph{Bernoulli}. vol. 6, no. 2, 339-355.

\bibitem{Nzi}
M. N'ZI (1997). Strassen's local law of the iterated logarithm for L$\acute{e}$vy's area. \emph{C. R. Acad. Sci. Paris}. vol. 324, no. 11, 1269-1273.

\bibitem{Ouahra}
M. Ouahra and M. Mellouk (2005). Strassen's law of the iterated logarithm for stochastic Volterra equations and applications. \emph{Stochastics.} vol. 77, no. 2, 191-203.

\bibitem{Peng}
L. Peng and Y. Qi (2003). Chover-type laws of the iterated logarithm for weighted sums. \emph{Stat. Probab. Lett.} vol. 65, no. 4, 401-410.

\bibitem{Priouret}
P. Priouret (1982). \emph{Remarques sur les Petites Perturbations de Syst$\acute{e}$mes Dynamiques}. \emph{S$\acute{e}$minaire de Probabilit$\acute{e}$s XVI}. Lecture Notes in Mathematics, vol. 920. Springer, Berlin.


\bibitem{Schied}
A. Schied (1997). Moderate deviations and functional LIL for super-Brownian motion. \emph{Stoch. Proc. Appl.}
vol. 72, no. 1, 11-25.

\bibitem{Strassen}
V. Strassen (1964). An invariance principle for the law of the iterated logarithm. \emph{Z. Wahrscheinlichkeistheorie}. vol. 3, no. 3, 211-226.

\bibitem{Wang}
F. Wang, J. Xiong, L. Xu (2016). Asymptotics of sample entropy production rate for stochastic differential equations. \emph{J. Stat. Phys.} vol. 163, 1211-1234.

\bibitem{Asymptotic}
R. Wang and L. Xu (2015). Asymptotics of the entropy production rate for d-dimensional Ornstein-Uhlenbeck processes. \emph{J. Stat. Phys.} vol. 160, no. 5, 1336-1353.

\bibitem{Wu}
L. Wu (1994). Large deviations, moderate deviations and LIL for empirical processes. \emph{Ann. Probab.} vol. 22, no. 1, 17-27.

\bibitem{J.Xiong}
 J. Xiong (2013). Super-Brownian motion as the unique strong solution to an SPDE. \emph{Ann. Probab.} vol. 41, no. 2, 1030-1054.

 \bibitem{J.Xiong2}
 J. Xiong (2013). \emph{Three Classes of Nonlinear Stochastic Partial Differential Equations}. World Scientific Publishing Co.

 \bibitem{Yamamuro}
 K. Yamamuro (2003). A law of the iterated logarithm of Chover type for multidimensional Levy processes. \emph{Osaka J. Math.} vol. 42, no. 2, 367-383.


\bibitem{Zhang}
M. Zhang (2009). On the weak convergence of super-Brownian motion with immigration. \emph{Sci. China Series A: Math.} vol. 52, no. 9, 1875-1886.

\bibitem{Zinchenko}
N. Zinchenko and A. Andrusiv (2008). Risk process with stochastic premiums. \emph{Theo. Stoch. Proc.} vol. 14, no. 3-4, 189-208.



\end{thebibliography}
\end{document}